\documentclass{elsarticle}
\usepackage{amssymb}
\usepackage{amsmath}
\usepackage{amsfonts}
\usepackage{amssymb}
\usepackage{amsthm}
\usepackage{float}
\usepackage{xspace}
\usepackage{epsfig}
\usepackage{xypic}
\usepackage{xy}
 \usepackage[colorlinks=true]{hyperref}
 \usepackage{color}

\hypersetup{urlcolor=blue, citecolor=red}
\xyoption{all}

\usepackage{hyperref}

\newtheorem{theorem}{Theorem}
\newtheorem{proposition}{Proposition}
\newtheorem{lemma}[proposition]{Lemma}
\newtheorem{remark}[proposition]{Remark}
\newtheorem{definition}[proposition]{Definition}
\newtheorem{corollary}[proposition]{Corollary}

\newtheorem{dictionnary}[proposition]{Dictionnary}

 \def\Q{\mathbb{Q}}

\def\id{{\mathrm{Id}}}  
\def\Id{\id}

 	\def\CF{\mathcal{F}}

\def\CM{\mathcal{M}}

\newcommand{\C}{\mathbb C}

\def\sym{\mathfrak{s}\mathrm{y}\mathfrak{m}} 
\def\Sym{\mathrm{Sym}}
\def\Const{\mathrm{Const}}
\def\const{\mathfrak{const}}
\def\kbar{\overline{{k}}}

\newcommand{\ggot}{\mathfrak{g}}

\newcommand{\hgot}{\mathfrak{h}}

\begin{document}
\title{A Characterization of Reduced Forms of Linear Differential Systems} 

\author[aam]{Ainhoa Aparicio-Monforte$^\star$\footnote{$\star$ \; Supported by the
Austrian FWF grant Y464-N18}}\ead{aparicio@risc.uni-linz.ac.at}
\author[ec]{Elie Compoint}\ead{Elie.Compoint@math.univ-lille1.fr}
\author[jaw]{Jacques-Arthur Weil}\ead{Jacques-Arthur.Weil@unilim.fr}
\address[aam]{RISC, Johannes Kepler University. Altenberger Strasse 69 A-4040 Linz, Austria.}
\address[ec]{D\'epartement de math\'ematiques, Universit\'e de Lille I, 59655, Villeneuve d'Ascq Cedex, France.}
\address[jaw]{XLIM (CNRS \& Universit\'e de Limoges) - 123, avenue Albert Thomas - 87060 Limoges Cedex, France}

\begin{abstract} 
A differential system $[A]\, : \; Y'=AY$, with $A\in \mathrm{Mat}(n, \overline{k})$  is said to be in reduced form if $A\in \mathfrak{g}(\overline{k})$
where $\mathfrak{g}$ is the Lie algebra of the differential Galois group $G$ of $[A]$.\\
In this article, we give a constructive criterion for a system to be in reduced form. When $G$ is reductive and unimodular, the system $[A]$ is in reduced form if and only if all of its invariants (rational solutions of appropriate symmetric powers) have constant coefficients (instead of rational functions). When $G$ is non-reductive, we give a similar characterization via the semi-invariants of $G$. In the reductive case, we propose a decision procedure for putting the system into reduced form which, in turn, gives a constructive proof of the classical Kolchin-Kovacic reduction theorem.
\end{abstract}
\maketitle

\section{Introduction}
The direct problem in  differential Galois theory is, given a linear differential system, how to compute its differential Galois group.
Though there exist several theoretical  procedures to achieve this task (Hrushovski \cite{Hr02a}, Compoint-Singer \cite{CoSi99a,PuSi03a}, van der Hoeven \cite{Ho07b}), it
is still far from being practical.
\\
If we now ask directly for the Lie algebra, not much is known. Nevertheless one can observe that if $Y'=AY$ is a differential system with
coefficients in $k=C(z)$ for example, then one can give a "supset" for
the Lie algebra $\ggot$ of the differential Galois group $G$. Indeed, Kolchin showed that $\ggot$ is
included in any Lie algebra $\hgot$ such that $A\in \hgot(k)$ (\cite{PuSi03a}, chap. 1).
If $A \in \hgot(k)$, and if we transform this system by a gauge matrix $P\in
GL_n({\bar k})$ in such a way that the new system $Y'=P[A]Y$ satisfies
$P[A] \in \tilde\hgot({\bar k})$ with $\tilde\hgot \subset \hgot$, then we obtain the 
finer majoration $\ggot \subset \tilde\hgot$.
The problem is then to find a matrix $P$ such that $P[A] \in \ggot({\bar k})$.
A result of Kolchin and Kovacic (proposition~\ref{prop: Kovacic-Kolchin} page \pageref{prop: Kovacic-Kolchin}) shows that it is always
possible to find such $P \in GL_n({\bar k})$. 
\\
When $P[A]\in \ggot({\bar k})$, we'll say that the system  is in reduced form.
But how can one recognize that a system is in reduced form?
In this paper we give an answer to that question. More precisely, we show
(theorem 1) that the system is in reduced form if and only if all its
semi-invariants can be written in a very particular form: they are all the
product of an exponential and a {\it constant coefficients vector} (instead of general rational coefficients, i.e. elements of $k$, as usual). We call them 
semi-invariants with constant coefficients
(see definition~\ref{inv-a-coeffs-csts}).
In the case where the Galois group $G$ is reductive and unimodular, this
criterion becomes (theorem \ref{thm: CWA}):
the system $Y'=AY$ is in reduced form if and only if all its invariants
have constant coefficients (instead of general rational coefficients). 
We then propose a procedure which puts a system into reduced form
(this also gives an algorithmic proof of the Kolchin-Kovacic reduction proposition~\ref{prop: Kovacic-Kolchin} page \pageref{prop: Kovacic-Kolchin}).\\
This work originates in the first author's PhD (\cite{Ap10a}) where a weaker version was presented. Other applications of reduced forms (to integrability of hamiltonian systems) are studied in \cite{ApWe11a,ApWe12a}.
\\

The paper is organized as follows.
In the first part we recall some basics on differential Galois theory.
We define, inspired by works of Wei and Norman, the Lie algebra associated to a matrix $A\in M_n(k)$ and the notion of reduced form.
In the second part,  we recall standard facts about the tensor constructions
of a differential module and the representations of the differential
Galois group associated with them.
In the third part we give and prove our main criterion to characterize when a linear differential system is in reduced
form.
The fourth part is devoted to the reductive unimodular case and the algorithmic aspects of these criteria.
We finish with examples illustrating the results of the paper.
\\

\noindent{\bf Acknowledgements}. We would like to thank P. Acosta-Humanez, E. Hubert and M.F Singer for fruitful discussions on this material. We also thank the referee for subtle and constructive comments.

\section{Reduced Forms of Linear Differential Systems}

\subsection{Differential Galois Group}
Let  $(k, \partial)$ denote a differential field of characteristic zero whose constant field $C$ is assumed to be algebraically closed.
The usual case is $k=C(x)$ with $\partial =\frac{d}{dx}$.
We consider a linear differential system $[A]\,:\, Y'=AY$ with $A\in\mathrm{Mat}(n,{k})$.
We refer to the book \cite{PuSi03a} for the differential Galois theory used here.
Let $K$ denote a Picard-Vessiot extension for $[A]$. 
Let $G:=Aut_\partial(K/k)$ be the differential Galois group of $[A]$ and $\mathfrak{g}=Lie(G)$ its Lie algebra. We also write $Lie([A])$ for the Lie algebra of the Galois group $G$ of $Y'=AY$.
The $C$-vector space of solutions of $[A]$ in $K^n$ is noted $V$. 
Remark that the selection of a fundamental solution matrix $U\in GL_n(K)$ of $[A]$ is equivalent with the choice of a basis $(Y_j)_{j=1\ldots n}$ of $V$.
\\
The differential module associated to $[A]$ is noted $\CM=(k^n,\nabla_{A})$.
The connexion $\nabla_A$ admits $A$ as its matrix in the canonical basis $(e_j)_{j=1\ldots n}$ of $k^n$ and is defined by
 $\nabla_A(Y)=\partial Y-AY$ for   $Y$ in $k^n$.
\begin{definition}
We say that two differential systems $Y'=AY$ and $Z'=BZ$, with $A,B \in \mathrm{Mat}(n,{k})$ are (gauge) equivalent (over $k$)  if there exists a linear change of 
variables (a gauge transformation) 
$Y=PZ$ with $P\in GL_n(k)$ changing $[A]$ to $[B]$, i.e.  $B=P[A]:=P^{-1}\cdot(A\cdot P-\partial(P))$.
\end{definition}
Changing system $[A]$ to an equivalent one is the same as changing basis in the differential module $\CM$.
Choosing a basis of $V$ yields a faithful representation of $G$ in $\mathrm{GL}(n,C)$. 
Hence there exists a polynomial ideal
$I\subset C[X_{1,1},\ldots ,X_{i,j},\ldots, X_{n,n},\frac1{Det}]$ (where $\frac1{Det}$ represents a solution $u$ of $u.\det(X_{i,j})=1$) such that 
$$ 
G\simeq \{M=(m_{i,j})\in \mathrm{GL}(n,C)\,:\, \forall P\in I \,,\, P(m_{i,j}) =0\}.
$$
We call  $I$ the \emph{ideal of relations of $G$}.
Similarly, we obtain a representation of $Lie(G)$ as 
$$
Lie(G):=\{N\in\mathrm{Mat}(n,C)\,:\, {\Id}+\varepsilon N \in G(C[\varepsilon])\,\text{with}\,\varepsilon \neq 0 \,\text{and}\, \varepsilon^2=0\}
$$
where ${\Id}$ is the identity matrix, and the  $C[\varepsilon]$-points of $G$ are the matrices $M_{\varepsilon}$ with coefficients in $C[\varepsilon]$ which satisfy all the 
equations of the ideal $I$.

\subsection{Definition and Existence of Reduced Forms}
Once we have fixed a representation of the differential Galois group $G$ (e.g. by choosing a basis of the solution space $V$ of $[A]$),
its Lie algebra  $\mathfrak{g}$ is a  $C$-vector space generated by matrices $\{M_1 ,\ldots , M_d \}\subset\mathrm{Mat}(n , C)$.
Another choice of basis yields a conjugate representation of   $\mathfrak{g}$  (see remarks below).
The set $\mathfrak{g}({k})$ of ${k}$-points of $\mathfrak{g}$  is 
$
\mathfrak{g}({k}):=\mathfrak{g}\otimes k = \{\, f_1 M_1 + \cdots + f_d M_d\,,\; f_i \in{k}\,\}.
$
\begin{definition}
Let $A\in \mathrm{Mat}(n ,\overline{{k}})$.
Consider the differential system  $[A]:Y' = AY$ with Galois group $G$ and its Lie algebra $\mathfrak{g}=Lie([A])$.
We say that $[A]$ is in \emph{reduced form} when $A\in\mathfrak{g}(\overline{{k}})$.
\end{definition}

The following classical result of Kolchin and Kovacic shows that any system admits a reduced form.

\begin{proposition}[Kolchin-Kovacic, \cite{PuSi03a} prop 1.31 \& cor 1.32]\label{prop: Kovacic-Kolchin}
Let ${k}$ be a $C_1$ field. 
Consider a differential system $Y'=AY$ with $A\in \mathrm{Mat}($n,{k}$)$. Assume that the differential Galois group $G$ (over ${k}$) 
is connected and let  $\mathfrak{g}$ be its Lie algebra. 
Let $H\subset GL_n(C)$ be a connected algebraic group whose Lie algebra $\mathfrak{h}$ satisfies 
	$A\in\mathfrak{h}({{k}})$. Then: 
\begin{enumerate}
\item $\mathfrak{g}\subset \mathfrak{h}$ and $G\subset H$.
\item There exists $P\in H({k})$ such that the equivalent differential system $Z'=\tilde{A}Z$, with $Y=PZ$ and $\tilde{A}=P[A]$, 
satisfies $\tilde{A}\in\mathfrak{g}({{k}})$, i.e. $\tilde{A}$ is in reduced form.
\end{enumerate}
\end{proposition}

\begin{corollary}
With the same notations, assume now that $G$ is \emph{not connected}. 
There exists $P\in H(\overline{{k}})$ such that the equivalent differential system  $Z'=\tilde{A}Z$, with $Y=PZ$ and $\tilde{A}=P[A]$, satisfies $\tilde{A}\in\mathfrak{g}(\kbar)$
\end{corollary}

\begin{proof} Let   $K$ be a Picard-Vessiot extension of ${k}$ for $Y' = AY$. 
By Galois correspondence (e.g.  proposition 1.34 p.26  in \cite{PuSi03a}), if we let $k^\circ:=\overline{{k}}\bigcap K$,
we have $G^{\circ}= Aut_{\partial}(K/k^\circ )$. 
As  $k^\circ$ is an algebraic extension of $ {k} $, it is still a $C_1$-field. 
Pick $k^\circ$ as a base field: then $K$ is a Picard-Vessiot extension of $k^\circ$  whose 
Galois group $G^{\circ}$ satisfies the conditions  of proposition  \ref{prop: Kovacic-Kolchin}.
\end{proof}

The assumption that  ${k}$ be a $C_1$ field is used only for the rationality issue (i.e. find a reduced form with coefficients in ${k}$ when $G$ is connected). 
If we allow a reduced form to have coefficients in an algebraic extension of ${k}$, then this hypothesis is not needed (we will come back to this and give a short proof in section  \ref{reduction-procedure}).

\subsection{Wei-Norman Decompositions and Reduced Forms} 

Let $A=(a_{i,j})\in\mathrm{Mat}(n,{k})$. We consider the $C$-vector space generated by the $n^2$ coefficients $a_{i,j}$. 
Let $a_1 ,\ldots , a_r$ be a basis of this space;  
each $a_{i,j}$ then admits a unique decomposition  $a_{i,j}=\sum^{r}_{k=1}m^{(k)}_{i,j} a_k$ with $m^{(k)}_{i,j}\in C$. 
We thus have, once the basis $(a_k)$ is chosen, a unique decomposition $A=\sum^{r}_{k=1} a_k M_k$ where $M_k =\left(m^{(k)}_{i,j}\right)_{i,j}\in\mathrm{Mat}(n,C)$

\begin{definition}
Let $A=(a_{i,j})\in\mathrm{Mat}(n,{k})$. A \emph{Wei-Norman decomposition} of $A$ is a decomposition  
$$
A=\sum^{r}_{k=1} a_k M_k,\quad M_k \in\mathrm{Mat}(n,C)
$$
where $\{a_1 , \ldots , a_r\}$ are a basis of $\mathrm{Span}_{C}((a_{i,j})_{i,j = 1,\ldots, n})$.
\end{definition}
  
Of course, this decomposition is not unique (it depends on the choice of the basis $\{a_1 , \ldots , a_r\}$) but its dimension $r$ is.
Note that the $C$-vector space generated by the $M_i$ (also of dimension $r$) depends only on $A$ (not on a choice of basis).
Indeed, consider an alternative basis $\{b_1 , \ldots , b_r\}$ of  $\mathrm{Span}_{C}((a_{i,j})_{i,j = 1,\ldots, r})$ such that 
$ a_j = \sum^{r}_{i=1} q_{ji} b_i$ with $Q:=(q_{i,j})_{i,j}\in\mathrm{GL}(r,C)$;
this gives an alternative Wei-Norman decomposition $\ A=\sum^{r}_{i=1} b_i  N_i$  with $N_i :=\sum^{r}_{j=1} q_{ji}M_j$.
It is clear that  $\mathrm{span}_{C}(M_1 , \ldots , M_r) = \mathrm{span}_{C}(N_1 , \ldots , N_r)$. 
\\

The notion of Wei-Norman decomposition borrows its name from Wei and Norman who, in \cite{WeNo63a,WeNo64a}, use the Lie algebra generated 
by the $M_i$ (presented there as the Lie algebra generated by all the $A(t)$ for $t$ spanning $\mathbb{C}$) to establish their method for solving linear differential systems.
\\

To extend their ideas to our context, we need a bit of classical terminology. The Lie algebra generated by a set of matrices $M_1,\ldots,M_r$ is the vector space generated by the $M_i$ and their iterated Lie brackets. A Lie algebra $\mathfrak{h}$ is called {\em algebraic} if there exists a connected linear algebraic group $H$ whose Lie algebra is $\mathfrak{h}$. The \emph{algebraic envelope} of a Lie algebra $\mathfrak{l}$ is the smallest algebraic Lie algebra containing $\mathfrak{l}$ ; if a Lie algebra is generated by given matrices, De Graaf and Fieker provide in \cite{FdG07} (section 3) an algorithm to compute its algebraic envelope
 
\begin{definition}\label{Lie algebra associated to the matrix A}
The \emph{Lie algebra $Lie(A)$ associated to the matrix $A\in\mathrm{Mat}(n,{k})$} is the algebraic envelope of the Lie algebra generated by the matrices $M_1,\ldots, M_r$ of a Wei-Norman decomposition of $A$.
\end{definition}
 

\begin{remark}
In definition~\ref{Lie algebra associated to the matrix A}, as $Lie(A)$ is required to be algebraic, it is the Lie algebra 
$\mathfrak{h}$ of some connected algebraic group $H$.  Hence $A\in \mathfrak{h}(k)$ and proposition~\ref{prop: Kovacic-Kolchin} applies, showing that $Lie([A])\subset Lie(A)$. The system is in reduced form when we reach the equality. 

Note that the Lie algebra generated by a Wei-Norman decomposition need not be algebraic; in such a case, it may not contain $Lie([A])$ and one really needs the algebraic envelope. 
The following example illustrates this fact. 

Consider the linear differential system of matrix $A:=\frac{1}{x}M_1$ where $M_1:=\left(\begin{array}{cc} \sqrt{2} & 0 \\ 0 & \sqrt{3}\end{array}\right)$. Then  $U:=\left(\begin{array}{cc} x^{\sqrt{2}}& 0 \\ 0 & x^{\sqrt{3}}\end{array}\right)$ is a fundamental matrix for $A$. 
Since $x^{\sqrt{2}}$ and $x^{\sqrt{3}}$ are algebraically independent, the Galois group $G$ of $A$ is $\C^{\star}\times \C^{\star}$, which has dimension $2$.

 The Lie algebra $Lie([A])$ of $G$ is generated by $\left\lbrace\left(\begin{array}{cc}1 & 0 \\ 0 & 0\end{array}\right)\,,\, \left(\begin{array}{cc} 0 & 0 \\ 0 & 1 \end{array}\right)\right\rbrace$ and has dimension $2$ as well. It is clear that $\{M_1\}$ yields a Wei-Norman decomposition of $A$, that $\mathrm{span}(M_1)$ has dimension $1$ and that  $Lie([A])\nsubseteq \mathrm{span}(M_1)$.
The reason why $Lie([A])$ is not included in $\mathrm{span}(M_1)$ is the following. The monodromy group of $[A]$ is generated by $\exp(2 \mathbf{i} \pi M_1)$ and  is not an algebraic group. What is algebraic is its Zariski closure which is $\left\lbrace\left( \begin{array}{cc} c_1 & 0 \\ 0 & c_2 \end{array}\right)\quad:\quad c_i \in \C^{\star}\right\rbrace$ and has dimension $2$. 
Note that $\mathrm{span}(M_1)$ is not an \emph{algebraic} Lie algebra.
\end{remark}

\begin{definition}\label{reduction-matrix}
$Y'=AY$ with $A \in Mat(n,k)$. A matrix $P\in GL_n(\bar{{k}})$ is called a {\em reduction matrix for $[A]$} if  $P[A]$ is in reduced form.
It is called a {\em partial reduction matrix for $[A]$} if $Lie(P[A])\subsetneq Lie(A)$.
\end{definition} 

Note that a reduction matrix $P$ may have coefficients in $\overline{{k}}\cap K$. It follows that a reduced form $R:=P[A]$
may have algebraic coefficients and $R\in\mathfrak{g}(\overline{{k}})$. As the change of variables $Y=PZ$ has algebraic coefficients, 
the Galois group may change but not its connected component of the identity, nor its Lie algebra.

\begin{remark}
Consider a system   $Y'=BY$ in reduced form. 
Then there exist elements $h_i\in\mathfrak{g}$, a basis $\{y_j\}$ of $V$ and a Wei-Norman decomposition  $B=\sum^{r}_{i=1} b_i  B_i$ of $B$ such that $B_i$ is the matrix of $h_i$ in the basis  $\{y_i\}$. 
Let $\{\tilde{y}_i\}$ be another basis of $V$, and let $Q \in
\mathrm{GL}(n,C)$ be the matrix of the change of basis from  $\{y_i\}$ to $\{\tilde{y}_j\}$. 
Then there exists a reduced form   $Z'=AZ$ in which the matrices $A_i$ of a Wei-Norman decomposition of  $A$ appear as matrices in the basis $\{\tilde{y}_i\}$
of elements $h_i$ of $\mathfrak{g}$. To see this, consider the matrix  $A=Q^{-1}BQ$. For each $i$,
$A_i=Q^{-1}B_i Q$ is the matrix of the element $h_i$ of $\mathfrak{g}$ in the basis $\{\tilde{y}_i\}$.
\end{remark}

Hence, having a reduced form  $Y'=AY$ associated with some \emph{a priori} given basis  $\{y_i\}$ of $V$ (i.e. $A=\sum^{r}_{i=1} f_i \cdot A_i$ 
where the  $A_i$ are elements  of $\mathfrak{g}$ represented in the basis $\{y_i\}$), we may obtain by conjugation a reduced form 
$Z'=RZ$ associated with any choice $\{\tilde{y}_i\}$ of basis of $V$.
In the sequel, we will 
apply this result  to solutions $\{\tilde{y}_i\}$ from a (local) fundamental solution matrix  $\hat{U}$ determined by the initial condition  $\hat{U}(z_0)={\Id}$ at some ordinary point $z_0\in C$ for $A$ (see section \ref{evaluation-morphism}).

We now prove a simple lemma useful in the proof of our criterion for reduced forms.
\begin{lemma}\label{crucial_for_criterion1}
Let $A=\sum_i f_iA_i$ be a Wei-Norman decomposition of $A$. 
Assume that $A$ admits a constant eigenvector $w\in C^n$  associated with an eigenvalue $f\in{k}$. Then $w$ is an eigenvector of each
$A_i$.
\end{lemma}
\begin{proof}
Let $w=\, {}^t(c_1,\ldots,c_n)\neq 0$. Assume, for example, that $c_1\neq 0$. Relation $Aw=f w$ gives us 
$f c_1 = \sum \alpha_i f_i$ with $\alpha_i\in C$, i.e. there exist constants $a_i$ such that $f=\sum a_i f_i$.
As $A w=f w$, we have $\sum_i f_i (A_i w) = \sum_i f_i (a_i w)$. As the $f_i$ are linearly independent over $C$, 
it follows that $A_i w = a_i w$ for all $i$ as claimed.
\end{proof} 


\section{Tensor Constructions on $V$ and on $\CM$}

Before turning to general tensor constructions, we start with the (simpler) case of the action of $G$ on symmetric powers of $V$ 
to ease the exposition.
The material in this section is mostly known to specialists but is here to set notations and clarify what follows.

\subsection{Actions of $G$ and $Lie(G)$ on $C[X_1 , \ldots , X_n]$}\label{subsec:actions}

\begin{definition}
Let $V$ denote a $C$-vector space of dimension $n$ and let $G\subset GL(V)$ be a linear algebraic group.
Let $(Y_i)$ be a basis of $V$. Let $\sigma\in G$ be an automorphism of $V$ and $M=(m_{i,j})\in{GL}(n,C)$ its matrix in the basis 
$(Y_i)$.
One defines an action (right translation) of $\sigma$ on $C[X_1, \ldots , X_n]$ by
$$
\sigma (X_j) = \sum^{n}_{i=1} m_{i,j}\cdot X_i .
$$
Given a homogeneous polynomial $P\in C[X_1 , \ldots , X_n]_{=m}$  of degree $m$, the action of  $\sigma$ on $P$ is defined by
$$
\sigma(P)(X_1 , \ldots ,X_n) := P(\sigma(X_1),\ldots , \sigma(X_n)).
$$
\end{definition}

We identify $P$ with its (column) vector  $v_P$ of coefficients in the symmetric power $\mathrm{Sym}^m (V)$.
Naturally, there exists a matrix, denoted by $\mathrm{Sym}^m (M)$, such that
$$
v_{\sigma(P)} =  \mathrm{Sym}^m (M)\cdot v_P .
$$
This way, the constructor $\mathrm{Sym}^m$ is a group morphism (from $GL(V)$ to $GL(Sym^m(V))$, called the 
\emph{$m$-th symmetric power in the sense of Lie groups}.
\begin{definition}\label{definition-sym}
In the above notations, let again $\mathfrak{g}=Lie(G)$.  To a matrix 
	$N:=(n_{i,j})\in\mathfrak{g}\subset \mathrm{Mat}(n,C)$, we associate the derivation
$$
D_N := \sum^{n}_{j=1}\left(\sum^{n}_{i=1} n_{i,j} X_i\right) \frac{\partial}{\partial X_j}.
$$
The action of $N\in\mathfrak{g}$ on elements  $C[X_1, \ldots , X_n]$ is defined as the action of the derivation  $D_N$ on the polynomials $P\in C[X_1, \ldots , X_n]$. 
\end{definition}
This may be understood using the $\varepsilon$-formalism. Let $\varepsilon$ be a variable subject to $\varepsilon^2=0$; 
we know (\cite{PuSi03a}, chap. 1,  or \cite{MiSi02a}, part 5) that $N\in \mathfrak{g}$ if and only if ${\Id} + \varepsilon N$ satisfies the group equations (modulo $\varepsilon^2$). If we set $\sigma_\varepsilon :={\Id}+\varepsilon N$, the above action (and the Taylor formula) induce 
	$\sigma_{\varepsilon} (P)=P+\varepsilon D_N (P)$, hence this definition.
As above, this defines a matrix $\sym^{m}(N)$ such that
$$
v_{D_N (P)}=\sym^{m}(N)\cdot v_P.
$$
We say that $\sym^{m}(N)$ is the \emph{$m$-th symmetric power of $N$ in the sense of Lie algebras}.
\begin{remark} The calculation following definition~\ref{definition-sym} shows that, for $B\in\mathrm{Mat}(n,{k})$,
an alternative characterization of  $\sym^{m}(B)$ is
	$$ \mathrm{Sym}^{m}({\Id} + \varepsilon B ) = {\Id}_N + \varepsilon\; \sym^{m}(B).$$
\end{remark}	
Similarly, we consider a linear differential system $[A]\,:\, Y'=A\cdot Y$ with $A\in\mathrm{Mat}(n,{k})$ and let $U$ 
be a fundamental solution matrix; 
using again a variable $\varepsilon$ satisfying $\varepsilon^2=0$ and using section \ref{evaluation-morphism} below to define evaluation of $U$ (if needed),  we have $U(x+\varepsilon)=U(x)+\varepsilon U'(x)=({\Id}+\varepsilon A)U(x)$ so 
$\Sym^m(U)(x+\varepsilon)=({\Id}+\varepsilon \; \sym^m(A))\Sym^m(U)(x)$. It follows that
$$
\mathrm{Sym}^{m}(U)' = \sym^{m}(A)\cdot \mathrm{Sym}^{m} (U).  
$$
\begin{lemma}
Consider the differential system $[A]\, :\, Y'=AY$ with $A\in\mathrm{Mat}(n,{k})$ 
and the change of variable $Y=PZ$ with $P\in{GL}(n,{k})$. Let $B=P[A]:=P^{-1}\cdot (A\cdot P - P')$.
Then, 
we have	$$\sym^m(P[A]) = \sym^m(B) = \mathrm{Sym}^m(P)[\sym^m(A)].$$ 
\end{lemma}
\begin{proof} Let $U$ and $V$ be fundamental solution matrices of  $[A]$ and $[B]$ respectively.  As $U=P\cdot V$ and 
$\mathrm{Sym}^m$ is a group morphism, we have $\mathrm{Sym}^m(U)=\mathrm{Sym}^m(P)\cdot\mathrm{Sym}^m(V)$
hence the result.
\end{proof}

\subsection{Tensor Constructions and Differential Modules}\label{sub : consdiff}
As above, $V$ denotes the solution space of $Y'=AY$  in a Picard-Vessiot extension $K$ of $k$ with the differential Galois group $G$ acting on $V$. 

\begin{definition} \label{tensor-construction}
A \emph{tensor construction} $\Const(V)$ on the $G$-module $V$ is a vector space obtained from $V$ by finite iteration of $\otimes$, $\oplus$, $\star$ (dual),  symmetric powers $\Sym^m$ and exterior powers $\Lambda^r$.
\end{definition}
\begin{remark}\label{chevalley}
In representation theory, one also considers subspaces and quotients. Here, these "direct" tensor constructions ($GL_n$-modules) will be enough for our purpose.
The proofs of the main results of this paper rely on Chevalley's theorem (\cite{Hu75a} thm 11.2, \cite{Bo91a} thm 5.1, \cite{Sp98a} thm 5.5.3) which states that there exists a tensor construction $\Const(V)$ and a line $D\subset \Const(V)$ such that $G$ is exactly the stabilizer of $D$ in $GL(V)$; the fact that the tensor constructions used here are enough can be seen by inspecting the proofs of Chevalley's theorem; for example in \cite{Be86a} page 67, it is shown that there is a finite dimensional  $W \subset C[End(V)] \simeq Sym(V\otimes V^\star)$ such that $D=\Lambda^r(W)$ (and hence lies in a tensor construction).
\end{remark}

The way $G$ acts on tensor constructions can be found e.g. in \cite{Hu75a, FuHa91a} or chapter 2 of  \cite{PuSi03a}; in particular, given a basis $Y=(Y_i)_{i=1\cdots n}$ of $V$ 
and a a tensor construction $\Const(V)$ on $V$, one may construct a canonical basis $\Const(Y)$ of $\Const(V)$.
 
Given $g\in G$ with matrix $M$, the matrix
$\Const(M)$ is defined (as in the above subsection) as the matrix of the action of $g$ on $\Const(V)$ relatively to this basis.
The constructor $\Const$ is then clearly a group morphism from $G$ to $GL_N(C)$ (with $N=\dim(\Const(V))$).
The associated "Lie algebra" constructor $\const$ may be (as above) defined by the identity
$\Const({\Id}+\varepsilon N)=\Id_N + \varepsilon\; \const(N)$, i.e. by the way a derivation $D_N$ (in the above notations of section \ref{subsec:actions}) acts on a construction.
This makes $\const$ a vector space morphism. 

\begin{lemma}\label{crucial_for_criterion2}
Let $A=\sum_i f_iA_i$ be a Wei-Norman decomposition of $A$. 
Let $\Const(\bullet)$ be a tensor construction. Then $\const(A)=\sum_i f_i \const(A_i)$ is a Wei-Norman decomposition of $\const(A)$.
\end{lemma}
\begin{proof}
Follows from linear independence (over $C$) of the $f_i$ and the fact that $\const$ acts linearly.
\end{proof}

Let now $\mathcal{M}:=({k}^n \,,\, \nabla_A = \partial -A)$ be the differential module. As shown in chapter 2 of  \cite{PuSi03a},
the matrix of the action of $\nabla_A$ on $\mathrm{Const}(\mathcal{M})$ is $\mathfrak{const}(A)$. In particular, in terms of a fundamental solution matrix $U\in{GL}(n,K)$, we have
 
$$
\mathrm{Const}(U)' = \mathfrak{const}(A)\cdot \mathrm{Const}(U).
$$
 
\subsection{A good fundamental solution matrix and its evaluation}\label{evaluation-morphism}

In the rest of this paper, we'll choose a convenient fundamental matrix $\hat{U}$. 
If $k=C(z)$ then we choose some ordinary point $z_0\in C$ 
of $[A]$ and let $\hat{U}$ denote the fundamental matrix of $[A]$ in $GL_n(C[[z-z_0]])$ such that $\hat{U}(z_0)={\Id}$. 

For a more general differential field, we can do more or less the same (we follow here the treatment in \cite{Br01a}, section 3). 
Let $k$ denote a finitely generated differential extension of $\bar{\Q}$. 
The $n^2$ coefficients of $A$ then lie in  a differential field $\tilde{k}$ which is a finitely generated differential extension of $\Q$.
By Seidenberg's Embedding Theorem \cite{Se58a,Se69a}, any such field
is isomorphic to a differential field $\CF$ of meromorphic functions on an open region
of $\C$.

So there exist infinitely many $z_0$ in $C$ such that all coefficients of $A$ can be seen to be analytic in an open neighborhood of $z_0$
(containing $z_0$). Such a $z_0$ will be called an ordinary point of $[A]$. We may then apply Cauchy's theorem to construct a local fundamental solution matrix $\hat{U}\in C[[z-z_0]]$ normalized by $\hat{U}(z_0)={\Id}$. 

Once  $\hat{U}$ is chosen, we first note that $\det(\hat{U})(z_0)=1$. 
Next  we recall that, for a tensor construction $\Const$, the entries of $\Const(\hat{U})$ are polynomials (with coefficients in $C$) in the entries of  $\hat{U}$ and $1/\det(\hat{U})$.    
As $\det(\hat{U})(z_0)=1$, the inverse of $\det(\hat{U})$ lies in $C[[z-z_0]]$ so that 
 the entries of $\Const(\hat{U})$ are polynomials in elements of $C[[z-z_0]]$.
As a consequence, the constructor $\Const$ commutes with evaluation at $z_0$ and, as $\Const$ is a group morphism from $GL_n(C)$ to $GL_N(C)$, $$\Const(\hat{U})(z_0)=\Const(\hat{U}(z_0))=\Const({\Id}_n)={\Id}_N.$$

In the rest of the paper, the reader may thus think of $k$ as being think of $k$ as being a subfield of $\C(\{x\})$ and then use this embedding to obtain the results for general $k$.
\begin{remark} \label{apparent-singularities}
Lemma \ref{crucial_for_criterion2} shows that, when one performs tensor constructions (in the sense of definition~\ref{tensor-construction}) on a differential system $[A]$, no apparent singularity may appear (this is not true if one performs constructions on operators). Indeed, the coefficients of $\const(A)$ will be linear combinations of the coefficients of $A$. Hence, if we choose an ordinary point $z_0$ of $[A]$, it remains an ordinary point of any tensor constructions $\const(A)$.
\end{remark}

\subsection{Invariants and Semi-Invariants of $G$ and of $\CM$}\label{invariants}

\begin{definition}
Let $\Const(V)$ denote a tensor construction on $V$. 
An element $I\in \Const(V)$ is called a \emph{semi-invariant} of $G$ if, for all  $g \in G$, 
$g(I)=\chi_{g}I$ where $\chi: G \rightarrow C^{*}$ is a character of $G$.
\\
We say that $I$ is an \emph{invariant} of $G$ if, for all  $g \in G$,  $g (I)=I$.
\end{definition}

For $f$ in $k$, we write $\exp(\int f)$ for a solution of $y'=fy$. A solution of $Y'=AY$ is called rational if $Y\in k^n$; it is called exponential if there exist $f\in k$ and $F\in k^n$ such that $Y=\exp(\int f)\cdot F$.

\begin{definition}\label{invariant-module}
A rational (resp. exponential) solution of some construction $Y'=\const(A)Y$ will be called an {\em invariant} (resp. {\em semi-invariant}\footnote{If $\phi_I=\exp(\int f)\; {}^t(\lambda_1,\cdots , \lambda_m)$ with $\lambda_i \in k$, then it is the vector $^t(\lambda_1,\cdots , \lambda_m)$ 
which lies in the module and could be called a semi-invariant of the module, so our notation is a bit abusive but coherent with the dictionary below.}) {\em of the differential module $\CM$}.
\end{definition}

As in section  \ref{evaluation-morphism}, pick an ordinary point $z_0\in C$ and 
let $\hat{U}$ denote the fundamental matrix of $[A]$ in $GL_n( C[[z-z_0]] )$ such that $\hat{U}(z_0)={\Id}$. 
To an element of $V$, represented by its vector $v$ of constants on the basis given by the fundamental matrix $\hat{U}$, we associate 
the element $\phi_v:=\hat{U}v$ of  $\CM\otimes K$; this gives the isomorphism 
$V\otimes K\rightarrow \CM\otimes K$. For a tensor construction $\Const$ on $V$, to a $v\in\Const(V)$ we associate the element
 $\phi_v:=\Const(\hat{U}).v\in \Const(\CM)\otimes K$.
For self-containedness, we recall  the following standard result.

\begin{lemma} In the above notations,
 $I$ is an invariant in $\Const(V)$ if and only if  $\phi_I$ is a rational solution of $[\const(A)]$, i.e. $\phi_I= {}^t(\lambda_1,\cdots , \lambda_m)$ with $\lambda_i \in{k}$, for all $i$.
\\
Similarly, $I$ is a semi-invariant in $\Const(V)$  if and only if $\phi_I$ is an exponential solution of $[\const(A)]$), i.e. $\phi_I= \exp(\int f)\; {}^t(\lambda_1, \cdots , \lambda_m)$ with $f\in k$ and $\lambda_i \in{k}$, for all $i$.
\end{lemma}
\begin{proof} Let $g\in G$. It acts on $\hat{U}$ by $g(\hat{U})=\hat{U}.M_g$. Given an element $I\in V$ represented by its vector $v$ of coefficients in $C$, the coefficient vector of $g(I)$ is $M_g.v$. 
If $I$ is an invariant in $\Const(V)$, then $\Const(M_g).v=v$ so $g(\phi_I)=\phi_I$, hence $\phi_I\in{k}^n$. 
Similarly, if $I$ is a semi-invariant, then $\Const(M_g).v= \chi_{g} v$ and $g(\phi_I)=  \chi_{g} \phi_I$. 
Pick a non-zero coordinate, say $f_1$, of $\phi_I$. As $g(f_1)=  \chi_{g} f_1$ for all $g\in G$, we have $f_1'/f_1\in{k}$.
Now, for any other coordinate $f_i$ of $\phi_I$, $g(f_i/f_1)=f_i/f_1$ so $f_i/f_1\in{k}$.
\end{proof}

This lemma gives the standard dictionary between rational (resp. exponential) solutions of $Y'=\const(A)Y$ and invariants (resp. semi-invariants) in 
$\Const(V)$. 
\begin{dictionnary}\label{dico}
If $v\in C^N$ is the vector of coefficients of an invariant $I\in\Const(V)$, the corresponding invariant of the module (rational solution of $Y'=\const(A)Y$)
is $\phi_I=\Const(\hat{U}).v$. 
\\
Conversely, given an invariant $\phi$ of the module (rational solution of $Y'=\const(A)Y$), we use the evaluation technique from section \ref{evaluation-morphism}: as $\hat{U}(z_0)={\Id}$, we will have $\phi=\Const(\hat{U}).v$ where $v=\phi(z_0)$ is the coefficient vector of the invariant $I\in\Const(V)$ associated  to $\phi$ (because $\Const({\Id})$ is again the identity). 
\end{dictionnary}
A similar dictionary is used e.g. in \cite{Br01a},\cite{CoSi99a}, \cite{HoWe97a}.
We may now introduce a key definition for our reduction criteria.

\begin{definition}\label{inv-a-coeffs-csts}
Let $[A]: Y'=AY$ with $A\in\mathrm{Mat}(n,{k})$ be a linear differential system. 
Let  $\mathcal{M}:=({k}^n \,,\, \nabla=\partial -A)$ be the associated differential module with basis $\{e_i\}$.

An invariant $\phi\in\mathrm{Const}(\mathcal{M})^{\nabla}$ of the module (a rational solution $\phi={}^t(\lambda_1,\ldots,\lambda_N)$
of $\phi'=\const(A)\phi$) will be said to have {\em constant coefficients} when $\lambda_i\in C$ for all $i=1,\ldots,N$.

A semi-invariant $\phi\in\mathrm{Const}(\mathcal{M})\otimes K$ of the module (an exponential solution 
$\phi_I=\exp(\int f)\cdot^t(\lambda_1,\cdots ,\lambda_N)$, with $f\in k$ and $\lambda_i \in k$, of $\phi'=\const(A)\phi$)
 will be said to have {\em constant coefficients} if $f$ can be chosen so that $\lambda_i\in C$ for all $i=1,\ldots,N$.

\end{definition}


\section{Characterization of Reduced Forms via their Semi-Invariants}
In this section, we give our main result, a simple criterion to constructively characterize  a system in reduced form.

\begin{theorem} \label{main-criterion}
Let ${k}$ denote a finitely generated differential extension of $\bar{\Q}$.
Let $[A]: Y'=AY$ with $A\in\mathrm{Mat}(n,{k})$ be a linear differential system. 
Let  $\mathcal{M}:=({k}^n \,,\, \nabla=\partial -A)$ be the associated differential module.
\\
The system $[A]$ is in reduced form if and only if any semi-invariant of the module $\CM$ 
has constant coefficients 
(in the sense of definitions \ref{invariant-module} and \ref{inv-a-coeffs-csts}).
\end{theorem}

\begin{proof}
As in section \ref{evaluation-morphism}, we fix a fundamental solution matrix $\hat{U}$ of $Y'=AY$ at an ordinary point $z_0$ such that $\hat{U}(z_0)={\Id}_n$.
This choice of basis for $V$ induces a matrix representation of $G$ and $\mathfrak{g}$ in $GL(V)$.
\\

Assume that $Y'=AY$ is in reduced form. We therefore have a Wei-Norman decomposition  $A=\sum f_iA_i$ 
($f_i$ in $k$, linearly independent over $C$) where each $A_i\in \mathfrak{g}$.
\\
Let $I$ be a semi-invariant in a tensor construction $\Const(V)$. We let  $v:=^t(\alpha_1,\cdots , \alpha_N) \in C^N$ 
be the coordinates of $I$ in $\Const(V)$ relatively to the basis induced by $\hat{U}$.
We let, as in dictionary \ref{dico}, $\phi_I:=\Const(\hat{U})\cdot v$. As $I$ is a semi-invariant, the associated exponential solution of $Y'=\const(A) Y$
is  $\phi_I=\exp(\int f)\cdot^t(\lambda_1,\cdots ,\lambda_N)$ with $f$ and the 
$\lambda_i$ in $k$. Of course,  as $\Const(\hat{U})(z_0)={\Id}_N$, we have
$\phi_I(z_0)=v$.

We will show that the $\lambda_i$ may in fact be forced to be constants. Recall that $\const(A)$ decomposes as 
$\const(A)=\sum_i f_i \const(A_i)$. As $A_i\in\mathfrak{g}$ and $v$ is the coordinate vector of the semi-invariant $I$, 
there exist constants $c_i$ such that
$\const(A_i) v = c_i v$. Define $\varphi :=\exp\left( \int \sum_i c_i f_i\right)$ and choose the integration constant such that $\varphi(z_0)=1$.
Let $F$ be the vector in $K^n$ defined by $F:=\varphi\cdot v$. As the $\alpha_i$ are constant, we have 
$$F'=\varphi'\cdot v = \varphi \left( \sum c_i f_i\right) v.$$
Now $$\const(A) \cdot F = \varphi\left( \sum_i f_i  \const(A_i) v \right) =  \varphi\left( \sum c_i f_i\right) v.$$
It follows that $F'=\const(A) \cdot F$. As $\varphi(z_0)=1$, we have $F(z_0)=v$. So $F$ and $\phi_I$ are two (exponential) solutions
of  $Y'=\const(A) Y$ with the same initial condition $\phi_I(z_0)=F(z_0)=v$ at an ordinary point $z_0$, hence $\phi_I=F$ and $\phi_I$ indeed has constant coefficients.
\\

We now prove the converse.  Assume that any semi-invariant of the system $Y'=AY$ has constant coefficients (in the sense of definition~\ref{inv-a-coeffs-csts}). 
By Chevalley's theorem (see remark \ref{chevalley}), there exist a tensor construction $\Const(V)$ and an element $I\in \Const(V)$ such that the algebraic group $G$
is exactly the set of automorphisms of $V$ which leave the line generated by $I$ stable. If we let $v$ denote the coordinates of 
$I$ in $\Const(V)$, then 
$$\mathfrak{g}=\left\{ M \in \mathrm{Mat}(n,C) \, | \exists c\in C, \, \const(M) v =c v\right\}.$$
As above,  let $\phi_I:=\Const(\hat{U})\cdot v$ be the associated exponential solution of $Y'=\const(A) Y$. Note that $\phi_I(z_0)=v$.
By hypothesis, there exist an exponential $\varphi$ (i.e. $\varphi\in K$ and $\exists f\in{k}$ such that $\varphi'/\varphi=f$) and a vector $w\in C^N$ of constants such that $\phi_I=\varphi w$. If we choose the integration constant such that $\varphi(z_0)=1$, we have $w=v$ (evaluate in $z_0$).
\\
Let $A=\sum_i f_iA_i$ be a Wei-Norman decomposition of $A$. Then $\const(A)=\sum_i f_i \const(A_i)$ is a Wei-Norman decomposition 
of $\const(A)$ (lemma  \ref{crucial_for_criterion2}). 
Now $\phi_I$ is a solution of $Y'=\const(A)Y$; we compute $$0= \phi_I'-\const(A)\phi_I = \varphi\left( f v - \sum_i f_i \const(A_i) v\right).$$
So $v$ is a constant eigenvector of $\sum_i f_i \const(A_i)$. By lemmas  \ref{crucial_for_criterion2} and \ref{crucial_for_criterion1}, it follows that $v$ is an eigenvector of each 
$\const(A_i)$ which in turn implies that, for all $i$, we have $A_i\in \mathfrak{g}$. Hence $A$ is in reduced form.
\end{proof}

A vector $v\in \Const(V)$ is  called a {\em Chevalley semi-invariant} for the group $G$ when $$G=\left\{ M \in{GL}(n,C) \, | \exists c\in C, \, \Const(M) v =c v\right\}.$$
The above proof shows the following finer corollary: 
\begin{corollary}
In the above notations, the system $[A]$ is in reduced form if and only if there is a Chevalley semi-invariant $v$ in a tensor construction $\Const(V)$ of $G$ such that 
$\phi_v:=\Const(\hat{U}) v$ has constant coefficients. 
\end{corollary}
Another  consequence of the proof of this theorem is that it may be adapted to a result for partial reduction.

\begin{corollary}
Consider a differential system $Y'=AY$ with $A\in \mathrm{Mat}(n,k)$.
Let $G$ be the differential Galois group and $H\supset G$ be some group containing $G$ with Lie algebra $\mathfrak{h}=Lie(H)$.
Then $A\in \mathfrak{h}(k)$ if and only if, for any semi-invariant $I$ of $G$ which is also a semi-invariant of $H$, its image $\phi_I$ has constant coefficients.
\end{corollary}
\begin{proof} The implication is proved like in the theorem. For the converse, replace $I$ in the above proof by a Chevalley semi-invariant defining $H$ (instead of $G$).
\end{proof}

We may illustrate this result with a very simple example. If $G\subset H:=\mathrm{SL}(n,C)$, then the determinant $\det$ of a fundamental solution matrix is an invariant - and this is the only invariant of $G$ which is an invariant of $H$. 
The equation that it satisfies is $\det'=Tr(A)\det$. 
Here, $\mathfrak{h}$ is the set of matrices with zero trace; we see that $A$ is in $\mathfrak{h}(k)$ if and only if $\mathrm{det}'=0$, i.e. if $\mathrm{det}$ is constant.

\section{A Reduction Procedure when $G$ is Reductive and Unimodular}
 
When $G$ is reductive and unimodular, we may refine our reduction criterion by reducing it to invariants in $\Sym(nV)$. Furthermore, we give a reduction procedure in this case (which also gives a constructive proof of the Kolchin-Kovacic reduction theorem). 
In what follows, we assume that $G$ is reductive and unimodular;  the term "Invariant" will thus refer specifically to an invariant in $\Sym(nV)$ (and an invariant of the module will be an invariant in  $\Sym(n\CM)^\nabla$).

\subsection{A Characterization of Reduced Forms via Invariants when $G$ is Reductive and Unimodular}

When $G$ is reductive, we start by recalling a useful variant of Chevalley's theorem (see also \cite{Bo91a}, chap.II \S 5.5, page 92)
\begin{lemma}\label{chevalley-reductif}
Let $G\subset GL_n(C)$ be a {\em reductive}   linear algebraic group acting on an $n$-dimensional vector space $V$.
There exists an invariant $I$ in a tensor construction on $V$ such that $G=\{g\in GL_n(C) \; | \; g(I)=I\}$.
\end{lemma}
\begin{proof}
By Chevalley's theorem (see  remark \ref{chevalley}), there exists a tensor construction $W$ on $V$ and a one dimension subspace ${\cal V}=C.v$ of $W$ such that 
$G=\{g \in GL_n(C) |  g(v)=\chi_g v, \chi_g\in C^*\}$. As $G$ acts completely reducibly, ${\cal V}$ admits a $G$-stable supplement $\widetilde{\cal V}$ in $W$ so that the dual $W^\star$ decomposes as a direct sum $W^\star = {\cal V}^\star \oplus \widetilde{\cal V}^\star$ of $G$-spaces and $G$ acts on $v^\star$ as
$g(v^\star)= 1/\chi_g v^\star$ (see e.g. \cite{We95a} lemmas 15 and 16 and their proofs). Then, $I:= v\otimes v^\star$ is an invariant of $G$ in $W\otimes W^\star$
and it is easily seen that $G=\{g\in GL_n(C) \; | \; g(I)=I\}$.
\end{proof}
Again, such an invariant $I \in \Const(V)$ giving $G=\{g\in GL_n(C) \; | \; g(I)=I\}$ will be called a {\em Chevalley invariant for $G$}.
For a reductive group $G$, the proof of theorem \ref{main-criterion} can be straightforwardly adapted (and we let the readers adapt its two corollaries to the case of invariants too).

\begin{proposition}\label{criterion-reductive}
Let $k$ denote a finitely generated differential extension of $\bar{\Q}$.
Let $[A]: Y'=AY$ with $A\in\mathrm{Mat}(n,{k})$ be a linear differential system. 
Let  $\mathcal{M}:=({k}^n \,,\, \nabla=\partial -A)$ be the associated differential module.
We assume that the differential Galois group $Gal([A])$ of $[A]$  is reductive.
\\
The system $[A]$ is in reduced form if and only if any {\em invariant} of the module $\CM$ 
has constant coefficients 
(in the sense of definitions \ref{invariant-module} and \ref{inv-a-coeffs-csts}) 
\end{proposition}

\begin{proof} The implication is proved as in theorem \ref{main-criterion} (invariants are semi-invariants) as follows. Let $I$ be an invariant in a tensor construction $\Const(V)$.  We let  $v:=^t(\alpha_1,\cdots , \alpha_N) \in C^N$ 
be the coordinates of $I$ in $\Const(V)$ relatively to the basis induced by $\hat{U}$ (from section \ref{evaluation-morphism}).
We let, as in dictionary \ref{dico}, $\phi_I:=\Const(\hat{U})\cdot v$. As $I$ is an invariant, the associated rational solution of $Y'=\const(A) Y$
is  $\phi_I=^t(\lambda_1,\cdots ,\lambda_N)$ with  the  $\lambda_i$ in $k$. Of course,  as $\Const(\hat{U})(z_0)={\Id}_N$, we have
$\phi_I(z_0)=v$.  Recall that $\const(A)$ decomposes as  $\const(A)=\sum_i f_i \const(A_i)$. As $A_i\in\mathfrak{g}$ and $v$ is the coordinate vector of the invariant $I$, we have 
$\const(A_i) v = 0$ so $\const(A) v = \sum_i f_i  \const(A_i) v =0$. As $v'=0$, we see that 
$v'=\const(A)v$. Now, as $\phi_I(z_0)=v$, we see that $v$ and $\phi_I$
 are two (rational) solutions
of  $Y'=\const(A) Y$ with the same initial condition at an ordinary point $z_0$, hence $\phi_I=v$ which has constant coefficients.
\\
To see the converse,
just use the above lemma \ref{chevalley-reductif} to replace "semi-invariant" by "invariant" and "exponential" by "rational" in the proof of the converse in theorem \ref{main-criterion}. 
\end{proof}

We now assume in the sequel that $G$ is reductive and unimodular and prove a finer version which restricts the type of constructions used then.
For this we recall some facts.
\\
Let $\Sym(nV)^{G}$ denote the ring of tensors of $\Sym(nV):=\Sym(V\oplus\cdots\oplus V)$ ($n$ copies of $V$) which are invariant 
under the action of $G$.
As $G$ is reductive this ring is finitely generated. Let $I_1,\cdots ,I_p$ be a 
system of generators of this ring.
As  $G$ is reductive and unimodular we can use \cite{Co96a} (or \cite{CoSi99a}) and convert 
this basis of $\Sym(nV)^{G}$ to a basis $P_1,\cdots , P_p$ of the ideal $J$ of 
algebraic relations with coefficients in $k$ satisfied by the elements of a 
fundamental matrix of solutions $\hat{U}$ of the differential system. We can define an 
action of $G$ on the polynomials of $k[X_{ij}]$ by setting
$\sigma P(X_{ij})= P(X_{ij}M_{\sigma})$ where $M_{\sigma}$ is defined by $\sigma 
(\hat{U})=\hat{U}M_{\sigma}$ as in section \ref{subsec:actions}.

With this action, the polynomials $P_j$ are invariant (because the tensors $I_j$ 
are).
The Galois group $G$ is then defined as the set of automorphisms of $V$ that 
preserve the ideal $J$ (\cite{CoSi99a} or \cite{PuSi03a}) 
and this is equivalent to say that $G$ is defined by
	$$G=\{ \sigma \in GL(V) \; | \; \sigma(I_k)=I_k, \; 1\leq k \leq p\}.$$ 
We also deduce (as in section \ref{subsec:actions}) that the Lie algebra
$Lie(G)$ of $G$ is defined by 
	$$Lie(G)=\{ g \in \mathfrak{gl}(V) \; | \;  D_g(I_k)=0, \; 1\leq k \leq p \}.$$

\begin{theorem}\label{thm: CWA} 
Let $k$ denote a finitely generated differential extension of $\bar{\Q}$.
Let $[A]: Y'=AY$ with $A\in\mathrm{Mat}(n,{k})$ be a linear differential system. 
Let  $\mathcal{M}:=({k}^n \,,\, \nabla=\partial -A)$ be the associated differential module.
We assume that the differential Galois group $Gal([A])$ of $[A]$  is reductive and unimodular.  
\\
The system $[A]$ is in reduced form if and only if any invariant $\phi\in\mathrm{Sym}(n\mathcal{M})^{\nabla}$ of the module 
has constant coefficients  (in the sense of definitions \ref{invariant-module} and \ref{inv-a-coeffs-csts}).
\end{theorem}

\begin{proof}
Assume that $[A]$ is in reduced form. By proposition~\ref{criterion-reductive}, all invariants have constant coefficients
so we only need to prove the converse implication.
Assume that, for any construction $\mathrm{Const}(\mathcal{M}) = \Sym^m(n\mathcal{M})$, any invariant 
$\phi\in\mathrm{Const}(\mathcal{M})^{\nabla}$ has constant coefficients (in the sense of definition~\ref{inv-a-coeffs-csts}).
Using section \ref{evaluation-morphism}, we choose an ordinary point $z_0$ and a fundamental solution matrix $\hat{U}$
verifying $\hat{U}(z_0)={\Id}$. 
The notation being as in the previous paragraph,  
 we take the columns of $\hat{U}$ as a basis of $V$ and express the coordinates $v_k$ of the generating invariants $I_k$ relatively to the induced basis in $\const_k(V):=\sym^{m_k}(nV)$.
 We thus have
$$
\mathfrak{g}=\{N\in\mathrm{Mat}(n,C)\; | \; \const_k(N)\cdot v_k = 0, \; 1\leq k\leq p \}.
$$
Let  $\mathcal{\phi}_k$ be the image of $I_k$ in $\Const_k(n\mathcal{M})^{\nabla}$ given by dictionary \ref{dico}, i.e
$\mathcal{\phi}_k=\Const_k(\hat{U})\cdot v_k$. As  $\mathcal{\phi}_k$ is constant by assumption
(and $\Const_k(\hat{U})(z_0)$ is the identity), the evaluation of this relation at $z_0$ gives $\mathcal{\phi}_k=v_k$. 

Now, as $\mathcal{\phi}_k'=\const_k(A)\cdot \mathcal{\phi}_k$ and 
$\const_k(A) = \sum^{r}_{i=1} f_i \cdot \const_k(A_i)$ is a Wei-Norman decomposition of $\const_k(A)$ 
(lemma \ref{crucial_for_criterion2}), we have $\sum^{r}_{i=1} f_i \; \const_k(A_i)\cdot v_k=0$. 
And, as the $f_i$ are linearly independent over $C$, it follows that $\const_k(A_i)\cdot v_k =0$ for all $i=1,\ldots ,r$ and for all $k=1,\ldots ,p$
which proves that $A_i\in\mathfrak{g}$ for $i=1,\ldots,r$ and the system is in reduced form.
 \end{proof}

 \begin{remark}
This proof sheds light on another useful feature (known to differential geometers, see e.g. \cite{DoNoFo85a}, Chap.6 \S 25 lemma p. 238). 
If the system is in reduced form, the above proof shows that  the canonical local solution matrix $\hat{U}$ satisfies  $v_k = \Const_k(\hat{U})v_k$ for $k=1,\ldots ,p$; as $G=\{g\in GL(V) \; | \;  g(I_k)=I_k, \; 1\leq k\leq p \}$  (see the discussion preceding the theorem), this means that the canonical local solution matrix $\hat{U}$
satisfies the equations of the group:  $\hat{U}\in G(K)$. 
\end{remark}

There remains a minor point in order to make this effective. Assume that we start from a matrix $A\in\mathrm{Mat}(n,{k})$. If we compute a reduction matrix $P$, it will have coefficients in some algebraic extension $k_0$ of $k$. So, to apply our theorem, we would need to compute solutions of 
the $\Sym^m(n\CM)\otimes k_0$ in $k_0^N$ and we would like to avoid that. This is provided by this simple but useful lemma.
\begin{lemma}\label{alg-to-rat}
Let $[A]:Y'=AY$ denote a linear differential system with $A\in\mathrm{Mat}(n,{k})$.  Assume that we know an algebraic extension $k_0$ of $k$ and a matrix 
$P\in GL_n(k_0)$ such that, for any invariant $\phi\in\Const(\CM)^\nabla$ (i.e. $\phi$ is a solution in $k^N$ of $\Const(\CM))$, $\Const(P)^{-1}\phi$ has constant coefficients.
Then, for any (algebraic) solution $\phi$ of any $(\Const(\CM) \otimes \bar{{k}})$, $\Const(P)^{-1}\phi$ also has constant coefficients. \\
In other words, $P[A]$ is in reduced form and $P$ is a reduction matrix.
\end{lemma}
\begin{proof} Take an algebraic solution $\phi$ of some construction $\Const$ (i.e. $\phi\in(\Const(\CM) \otimes \bar{{k}})^\nabla$) and choose an entry $f$ 
of this solution.
The coefficients of the 
minimal polynomial
of $f$ are given as entries in ${k}$ of a solution of another tensor construction: as shown e.g. by Singer and Ulmer \cite{SiUl93a,SiUl97a}, 
coefficients of minimal polynomials of algebraic solutions of linear differential equations are found as invariants in symmetric powers.
As $P^{-1}$ maps the latter to constants, the image of $f$ under $P^{-1}$ will have a minimal polynomial whose coefficients are constants. 
As $C$ is algebraically closed, this in turn implies that the image of $f$ is in $C$, hence the result.
\end{proof}

\subsection{The Compoint-Singer Procedure for Computing $G$ when $G$ is Reductive}
We now assume that ${k}$  is an algebraic extension of $C(x)$.
Theorem 1.1 of Compoint in  \cite{Co96a,Co98a} shows that if a Picard-Vessiot extension $K$ has a reductive unimodular Galois group $G$,
then the generators of the ideal of relations among solutions (and their derivatives) are obtained from the generators of 
$\mathrm{Sym}(n\mathcal{M})^{\nabla}$. From van Hoeij and Weil \cite{HoWe97a} we know that, for a given $m$, 
the generators of degree $m$ of $\mathrm{Sym}(n\mathcal{M})^{\nabla}$ may be constructed directly from system $[A]$.
So computing $G$ is reduced to computing a bound on the degree of the generators of the ring of invariants of $G$.
The work of  Compoint and Singer \cite{CoSi99a} shows how to compute such a bound and culminates in this result:

\begin{proposition}[Compoint, Singer \cite{CoSi99a}] 
Let ${k}$ denote an algebraic extension of $C(x)$, where $C$ is an algebraically closed computable field.
Let $[A]:Y'=AY$ be a linear differential system with $A\in \mathrm{Mat}(n,{k})$ and a reductive unimodular Galois group.
Then one can compute in a finite number of steps a basis $\mathcal{\phi}_1,\ldots ,\mathcal{\phi}_r\in {k}[Y_{i,j}]$ of
$\mathrm{Sym}(n\mathcal{M})^{\nabla}$. 
\end{proposition}
From an implementation viewpoint, this procedure is far from being efficient, but it gives a (yet somewhat theoretical) starting point
for our reduction algorithm below.

\subsection{A Reduction Procedure for Reductive Unimodular Differential Systems}\label{reduction-procedure}
We assume, as above, that ${k}\subset\overline{C(z)}$ and that $G:=Gal([A])$ is reductive and unimodular. 
Continuing with our notations, we consider $\hat{U}\in GL_n(\,C[[z - z_0 ]]))$ a fundamental solution matrix of $[A]$ satisfying the initial condition $\hat{U}(z_0)={\Id}$ at an ordinary point $z_0$ of $[A]$.
At this point, we have a simple procedure for computing a reduced form.
\\

\noindent {\bf Reduction procedure}
\begin{enumerate}
\item Choose a regular point $z_0$ in $C$ and a canonical local solution matrix $\hat{U}\in GL_n(\,C[[z - z_0 ]])$ satisfying $\hat{U}(z_0)={\Id}$.
\item Using the Compoint-Singer procedure, compute a set of invariants $ \mathcal{\phi}_1,\ldots,\mathcal{\phi}_r$ in  $\mathrm{Sym}(n\mathcal{M})^\nabla$
and the set of invariants  $I_i := \mathcal{\phi}_i(z_0)$ of $G$ in $\Sym^{m_i}(nV)$ such that $G=\{g\in GL(V)| \forall i, g(I_i)=I_i\}$.
\item Pick a matrix $P$ with indeterminate coefficients. Define a system $(S)$ of polynomial equations 
	$$(S):\; \forall i =1\ldots,r , \quad \Sym^{m_i}(P\oplus\cdots\oplus P)\cdot I_i - \mathcal{\phi}_i = 0 , \det(P)\neq 0.$$
\item Using any algorithm for polynomial system solving (Groebner bases, triangular sets), find a solution in $GL_n(\bar{{k}})$ of $(S)$. 
The resulting $P[A]$ is in reduced form.
\end{enumerate}

\begin{theorem}
The above algorithm is correct and computes a reduced form for any system with a reductive unimodular Galois group.
\end{theorem}

\begin{proof} Recall that an algorithm is correct when each step is well defined, it terminates in finite time, and the result is what was expected.\\
The fact that Step 1 makes sense follows from section \ref{evaluation-morphism}. That Step 2 is well defined follows from \cite{CoSi99a} (using  \cite{Co96a,Co98a} 
and \cite{HoWe97a}). Step 4 is well defined if we can guarantee that system $(S)$ is always consistent. But, by construction, the local fundamental solution matrix 
$\hat{U}$ satisfies $(S)$. 
As $(S)$ has this solution in $K^{n^2}$, by Hilbert's Nullstellensatz, it has a solution in ${\bar{{k}}}^{n^2}$ (which can be obtained by Groebner bases \cite{CoLiOS07a} 
or triangular sets (e.g. \cite{Hu03a})). Now our theorem  \ref{thm: CWA} and lemma \ref{alg-to-rat} 
show that the resulting $P[A]$ is in reduced form. 
Indeed, $P$ sends the $\phi_i$ to invariants $I_i$ with constant coefficients and, by assumption, any invariant is polynomial in the $\phi_i$ and is hence mapped to one with invariant coefficients.
\end{proof}

\begin{remark}
This algorithm computes a reduction matrix with coefficients in ${\bar{{k}}}$. 
However, when $G$ is connected, the Kolchin-Kovacic reduction theorem (proposition~\ref{prop: Kovacic-Kolchin})  shows that there should exist a solution 
$P$ with coefficients in ${k}$. This is where the $C_1$ condition on {k} enters. 
If the equations defining $P$ are quadrics over {k}, then there exists an algorithm to find a solution in {k} (\cite{CrHo06a}). In general, though the 
proof of Tsen's theorem can be made constructive, we do not know of an implemented algorithm to determine a rational point on an algebraic manifold on ${k}$ 
in full generality. So the above algorithm, though correct and complete, does not fulfill the Kovacic program completely in the sense that it does not yet address 
the rationality problem "find $P\in GL_n(k)$" in the case when $G$ is connected.
\end{remark}

Apart
from the rationality issue, this theorem gives us a constructive proof of the classical Kolchin-Kovacic reduction theorem (proposition 
\ref{prop: Kovacic-Kolchin}) for reductive unimodular groups. We know (e.g. from section 5 of \cite{MiSi02a}) that, if $A\in \hgot(k)$, then there exists 
a fundamental solution matrix in 
$H(K)$ so that $(S)$ will have a solution in $H(K)$  as in the above proof. 

To obtain a similar proof of the reduction theorem for general groups (dropping the $C_1$ assumption on the fields and the rationality issue), 
pick a Chevalley semi-invariant $\phi_I$ exponential solution of some $\const(A)$ as in the proof of theorem  \ref{main-criterion} 
and write $\phi_I=exp(\int\varphi)V$ with $exp(\int\varphi)(z_0)=1$ and $V\in{k}^N$; 
letting $I:=V(z_0)$, the same reasoning shows that, for a $P$ with unknown coefficients, the system $V=\Const(P).I$ has a solution $P$ in $GL(\bar{{k}})$ 
(because $\hat{U}$ satisfies these equations so they are consistent) and the proof of theorem  \ref{main-criterion} then shows that $P[A]$ is in reduced form. 

If the group is connected and we have a reduced form with coefficients in $\bar{{k}}$ then to find one with coefficients in ${{k}}$ is a descent question 
(see e.g.  \cite{CoPuWe10a,CoWe04a}).
\\

We now turn to the structure of the solutions of the polynomial system $(S)$ of the reduction procedure. We first note a simple lemma.
\begin{lemma} Let $k$ be as above and $A\in Mat(n,k)$.
Let $k_0$ be an algebraic extension of $k$ and $B\in Mat(n,k_0)$ such that the system $Y'=BY$ is a reduced form of $[A]$. 
\begin{enumerate}
\item
The differential Galois group of $[B]$ is connected.
\item 
If there is a reduction matrix $P\in GL_n(k)$ such that $B=P[A]$, then $k_0=k$ and $Gal([A])$ is connected.
\end{enumerate}
\end{lemma}
\begin{proof}
We assume that $k_0$ is the coefficient field of $B$, i.e. the smallest field containing all coefficients of $B$. Let $K$ be a Picard-Vessiot extension of $k_0$ associated to $[B]$ and $G=Gal(K/k_0)=Gal([B])$ be the differential Galois group with Lie algebra $\ggot$.
As $B\in\ggot(k_0)$, Kolchin's theorem (proposition 1.31 in \cite{PuSi03a}, take $H=G^\circ$) shows that $G\subset H:=G^\circ$, hence $G=G^\circ$ and $G$ is connected. 
If now $B=P[A]$ with $P\in GL_n(k)$, then the coefficient field of $B$ is inside $k$ and $Gal([B])=Gal([A])$ so (2) follows from (1).
\end{proof}
If $k$ is a $C_1$ field then this lemma and the Kovacic-Kolchin reduction theorem (proposition~\ref{prop: Kovacic-Kolchin}) show that $Gal([A])$ is connected 
if and only if there exists a reduction matrix $P$ in $GL_n(k)$.

\begin{proposition}
Consider the polynomial system $(S)$ of the reduction procedure. 
Let $k_0:=K^{G^\circ}$ denote the fixed field of the connected component $G^\circ$ of $G$ in $K$. 
\begin{enumerate}
\item If $k$ is a $C_1$ field, then there exists a solution to $(S)$ in $GL_n(k_0)$. Furthermore, $k_0$ is the smallest field over which a solution can be found.
\item Suppose we have a solution $P\in GL_n(\bar{{k}})$ of $(S)$.  A matrix $P\cdot M$ is another solution of $(S)$ if and only if $M\in G(\bar{{k}})$.
\item A matrix $P\in GL_n(\bar{{k}})$ is a solution of $(S)$ only if $P'\cdot P^{-1}=A + N$ with $N\in \mathfrak{g}(\bar{k})$.
\end{enumerate}
\end{proposition}

\begin{proof}
View the differential system as having coefficients in $k_0$. Then its Galois group is $G^\circ$ and is connected. 
So, by proposition~\ref{prop: Kovacic-Kolchin}  of Kolchin and Kovacic, there exists a reduction matrix in $GL_n(k_0)$, hence Part 1. The fact that $k_0$ is minimal follows from the previous lemma.
\\
For part 2, we have $\Const(PM).I=\phi=\Const(P).I$ so, as $\Const$ is a group morphism, $\Const(M).I=I$ for all invariants defining $G$. This proves that 
	$M\in G(\bar{{k}})$. Part 3 is the derivative of this relation. Let $B:=P'\cdot P^{-1}$ 
We have $\phi'=\const(A)\phi$ and $$\phi'=\Const(P)' I = \const(B)\Const(P) I= \const(B) \phi$$ so, as the constructor $\const$ is linear, $\const(B-A)\phi=0$ and $B$ is given as the solution of a linear (non-differential) system.
\end{proof}

Note that even the reduction matrix $P$ has coefficients in a big algebraic extension, the reduced form $P[A]$ may still be defined over a small field.
For example, consider a system whose solutions are all algebraic. Then $P$ is just a fundamental solution matrix (algebraic) and $P[A]$ is the zero matrix (which has coefficients in a rather smaller field).

\section{Examples}
\subsection{A Dihedral Case in $GL_2(C)$}
$$A= \left( \begin {array}{cc} 0&1\\ \noalign{\medskip}x& \frac{1}{2x}  
\end {array} \right).$$
A variant of Kovacic's algorithm shows that the Galois group is a central extension of $D_\infty$ the infinite dihedral group.
Indeed, it has an invariant in $\Sym^2(\Lambda^2(\CM))$ (with value $x$) and and invariant in $\Sym^2(\CM)$ with value
${}^t(-1,0,x)$. If we choose the evaluation point $z_0=1$, the corresponding invariants are $I_1=1$ and $I_2=(-1,0,1)$.
Following our procedure, the reduction matrix $P$ should satisfy
$$  \left\{ x- \left( p_{1,1}\,p_{2,2}-p_{1,2}\,p_{2,1} \right)^{2},
2\,p_{1,1}\,p_{2,1}-2\,p_{1,2}\,p_{2,2},-1+{p_{1,1}}^{2}
-{p_{1,2}}^{2},x+{p_{2,1}}^{2}-{p_{2,2}}^{2} \right\}  $$
A triangular set computation gives simpler equations 
	$(p_{1,1},1+{p_{1,2}}^{2},x+{p_{2,1}}^{2},p_{2,2})$ 
hence the reduction matrix
$$P:=  \left[ \begin {array}{cc} 0&i\\ \noalign{\medskip}i\sqrt {x}& 0
\end {array} \right] 
$$
and 
$$P[A]=\sqrt {x} \left[ \begin {array}{cc} 0& 1\\ \noalign{\medskip}1& 0
\end {array} \right] $$
If we want to recover the more standard diagonal representation, we conjugate again by 
$$P_2:= \left[ \begin {array}{cc} 1& -1\\ \noalign{\medskip}1& 1\end {array}
 \right] $$
 which turns our invariant $[-1,0,1]$ into $(0,1,0)$ and 
 $$(P.P_2)[A]= \sqrt {x} \left( \begin {array}{cc} 1 &  0\\ \noalign{\medskip} 0 &  -1 \end {array} \right)$$
which is obviously a reduced form. Note that, in this example, it is not possible to obtain a reduced form without extending the ground field.

\subsection{Example: the Reduction Method when  $Lie([A])=\mathfrak{so}(3)$}
The special orthogonal group $\mathrm{SO}(3)$ is defined as
$$
\mathrm{SO}(3):=\{\, g\in{GL}_{3}(C)\,:\, g\in\mathrm{stab}(Y^2_1 + Y^2_2 + Y^2_3)\quad et \quad \mathrm{det}(g)=1\}.
$$
We consider $A\in\mathrm{Mat}(3,{k})$ and show how to put $[A]:Y'=AY$  
in reduced form with our procedure when its Galois group is $SO(3)$.
Note that as $SO(3)$ is conjugate to a symmetric square of $SL(2)$, this is an instance of the old problem of solving linear differential
equations in terms of equations of lower order (see \cite{Si85a}, \cite{Pe02a}\cite{Ho07a,HoPu06a}, \cite{Ng09a,NgPu10a}).

Testing whether $G\subset \mathrm{SO_{3}}(C)$ or not is achieved in three steps (see \cite{SiUl93a,HoRaUlWe99a}):

\begin{enumerate}
\item To check whether $G\subset\mathrm{SL}(3\,,\, C):=\{\,g\in{GL}(C,3)\,:\, \mathrm{det}(g)=1\}$, check whether there exists
$w\in {k}$ such that 
$w' = \mathrm{Tr}(A)w$. If yes, apply to $A$ the gauge transformation:
$$
P:=\left[\begin{array}{ccc}\omega &  0 &  0 \\ 0 &  1 &  0 \\ 0 &  0 &  1\end{array}\right]
$$  
so that $Tr(P[A])=0$, i.e. $P\in\mathfrak{sl}(3)$. Remark that if $U$ is a fundamental solution matrix of $[A]$, then 
$V:=P^{-1}U$ will be a solution matrix of $Z'=P[A]Z$ such that $\mathrm{det}(V)$ is constant, in the spirit of our criterion.
\item Using any factorization algorithm, check whether $[A]$ is irreducible (if it is reducible, use \cite{ApWe12a}).
\item Find an invariant of $[A]$ of degree $2$, i.e. test whether $[\sym^{2}(A)]$ has rational solutions. If it does not, then 
$\mathfrak{g}    \not\subset \mathfrak{so}(3)$. 
Otherwise, we have a solution
$Y={}^t (f_1 , f_2 , f_3 , f_4 , f_5 , f_6 )\in{k}^6$ of $Y'=\sym^{2}(A)Y$. 
\end{enumerate} 
Now, to achieve reduction, we find a gauge transformation $Q$ which transforms $Y$ to the form ${}^t (1,0,0,1,0,1)$.  
Indeed, if we let $\hat{U}$ be a solution matrix satisfying $\hat{U}(z_0) = {\Id}_n$ then $Y=\mathrm{Sym}^{2}(\hat{U})\cdot I$ 
and, by Gauss reduction of quadratic forms, $I$ is conjugate to $(1,0,0,1,0,1)$ i.e. $X^2_1 + X^2_2 + X^2_3$.

We apply this to the system given by the matrix
$$
A=\left[\begin{array}{ccc}\frac{2x^2 - 2x +1}{x(-1+x^2)} &  \frac{5x - 3x^3 + 2x^5 - 1+ x^2 - x^4}{(x-1)x}&  -\frac{2x^4 - 3x^3 + x +2}{x^2 (x-1)^2}\\
-\frac{x(2x-1)}{(x+1)(-1+x^2)} &  -\frac{x^5 - x^4 - x^3 + x^2 + 4x -1}{-1+x^2}&  \frac{x^4 - 2x^3 + 2x^2 + 1}{x(x+1)(x-1)^2 }\\
-\frac{x^2 (x-1)}{x+1}&  -x(x-1)(1-x^2 + x^4) &  \frac{x^5 - 2x^4 + x^3 + 2x - 1}{x(x-1)}\end{array}\right]
$$
We check that $[A]$ is irreducible (e.g. with algorithms in {\sc Maple}). The system admits an invariant of degree $2$, i.e. a rational solution of $[\sym^2(A)]$, namely
$
I_2:=-\frac{(3-x^2+x^4 - 2x)}{x^2 (x-1)^2}X^2_1 + 2\frac{(x^2 - 2x + 2)}{x(x+1)(x-1)^2} X_1 X_2 + (2x-2) X_1 X_3 - \frac{(x^2 - 2x +2)}{(x+1)^2 (x-1)^2} X^2_2 - 2\frac{x(x-1)}{x+1}X_2 X_3 - x^2 (x-1)^2 X^2_3.
$ (which can be found using \cite{HoWe97a} or the Barkatou algorithm\footnote{A {\sc Maple} implementation is available at
\url{http://perso.ensil.unilim.fr/~cluzeau/PDS.html}, see \cite{BaClElWe12a}.} applied to $\sym^2(A)$). Using standard Gauss reduction of quadratic forms, the invariant $I_2$ becomes:
$$
(-3+x^2 - x^4 + 2x)Z^2_1 - \frac{(x^2 - 2x +2)}{3-x^2 + x^4 -2x}Z^2_2 - \frac{1}{x^2 - 2x + 2}Z^2_3.
$$
We now apply an algorithm for polynomial solutions of quadric equations (e.g. \cite{CrHo06a}) which yields the polynomial change of variables  
$Z=P_2 \cdot Y$ defined by
$$
P_2 := \left[\begin{array}{ccc}(x+1)(x-1) &  x^2 -2x+2 &  0 \\ -1 &  (x+1)(x-1) &  1-x\\ 1-x &  (x+1)(x-1)^2 &  1\end{array}\right].
$$
So, finally, we combine this into transformation $(X_1 , X_2 , X_3) = (Y_1 , Y_2 , Y_3).P$ with
$$
P:=\left[\begin{array}{ccc}\frac{x(x-1)}{x^2 -1} &  1 &  0 \\ 0 &  x^2 -1 &  -\frac{1}{x}\\ 0 &  0 &  \frac{1}{x(x-1)}\end{array}\right] \quad \text{ and we obtain }\quad P[A] =\left[\begin{array}{ccc} 0 &  x &  1\\  -x &  0 &  x^2 \\ -1 &  - x^2 &  0\end{array}\right]\in\mathfrak{so}(3)({k}).
$$

%
\bibliographystyle{amsalpha}
\providecommand{\bysame}{\leavevmode\hbox to3em{\hrulefill}\thinspace}
\providecommand{\MR}{\relax\ifhmode\unskip\space\fi MR }
\providecommand{\MRhref}[2]{%
  \href{http://www.ams.org/mathscinet-getitem?mr=#1}{#2}
}
\providecommand{\href}[2]{#2}

\end{document}